\documentclass[12pt,reqno]{amsart}
  \usepackage{latexsym} 
  \usepackage[all]{xy}
  \usepackage{amsfonts} 
  \usepackage{amsthm} 
  \usepackage{amsmath} 
  \usepackage{amssymb}
  \usepackage{pifont}  
  \usepackage{enumerate}
  \usepackage{dcolumn}
  \usepackage{comment}
\newcolumntype{2}{D{.}{}{2.0}}
  \xyoption{2cell}

 \usepackage{tikz}
\usetikzlibrary{arrows}

  \def\<{{\langle}} 
  \def\>{{\rangle}}

  \def\note#1{{}}

  \def\note#1{} 

  \def\End#1#2{{{\rm End}\sb{#1}(#2)}}

  \def\beq{\begin{equation}} 
  \def\eeq{\end{equation}}

  \def\id{\mathrm{id}}

  \newcounter{zlist} 
  \newenvironment{zlist}{\begin{list}{(\arabic{zlist})}{ 
  \usecounter{zlist}\leftmargin2.5em\labelwidth2em\labelsep0.5em 
  \topsep0.6ex
  \parsep0.3ex plus0.2ex minus0.1ex}}{\end{list}}

  \newcounter{blist} 
  \newenvironment{blist}{\begin{list}{(\alph{blist})}{ 
  \usecounter{blist}\leftmargin2.5em\labelwidth2em\labelsep0.5em 
  \topsep0.6ex 
  \parsep0.3ex plus0.2ex minus0.1ex}}{\end{list}} 

  \newcounter{rlist}

\def\stac#1{\raise-.2cm\hbox{$\stackrel{\displaystyle\otimes}{\scriptscriptstyle{#1}}$}}

\def\cten#1{\raise-.2cm\hbox{$\stackrel{\displaystyle\widehat{\otimes}}
{\scriptscriptstyle{#1}}$}}

  \def\Label#1{\label{#1}\ifmmode\llap{[#1] }\else 
  \marginpar{\smash{\hbox{\tiny [#1]}}}\fi} 
  \def\Label{\label}

  \newtheorem{proposition}{Proposition}[section]
  \newtheorem{lemma}[proposition]{Lemma} 
  \newtheorem{corollary}[proposition]{Corollary} 
  \newtheorem{theorem}[proposition]{Theorem} 
  
\theoremstyle{definition} 
  \newtheorem{definition}[proposition]{Definition}
    
  \newtheorem{example}[proposition]{Example}

  \theoremstyle{remark} 
  \newtheorem{remark}[proposition]{Remark}

  \newcounter{c} 
   
  \newcommand{\etyk}[1]{\vspace{-7.4mm}$$\begin{equation}\Label{#1} 
  \addtocounter{c}{1}} 
  \renewcommand{\]}{\ifnum \value{c}=1 $$\else \end{equation}\fi} 
   \numberwithin{equation}{section}

\newcommand{\Ker}{{\rm Ker}\,}

\def\FF{{\mathbb F}}

\def\NN{{\mathbb N}}

\def\QQ{{\mathbb Q}}

\def\ZZ{{\mathbb Z}}

\newcommand{\oO}{\mathrm{O}}

\newcommand{\tT}{\mathrm{T}}
\newcommand{\uU}{\mathrm{U}}

\newcommand{\Cc}{\mathcal{C}}

\def\*C{{}^*\hspace*{-1pt}{\Cc}}

\def\text#1{{\rm {\rm #1}}}

 \def\1{\mathbf{1}}

\def\id{\mathrm{id}}

\def\lto{\longmapsto}
\def\lra{\longrightarrow}

\def\1\mathbf{1}

\begin{document}

\title[Congruence classes and extensions of rings]{On congruence  classes and extensions of rings with applications to braces}

\author{Tomasz Brzezi\'nski}

\address{
Department of Mathematics, Swansea University, 
Swansea University Bay Campus,
Fabian Way,
Swansea,
  Swansea SA1 8EN, U.K.\ \newline \indent
Department of Mathematics, University of Bia{\l}ystok, K.\ Cio{\l}kowskiego  1M,
15-245 Bia\-{\l}ys\-tok, Poland}

\email{T.Brzezinski@swansea.ac.uk}

\author{Bernard  Rybo{\l}owicz}

\address{
Department of Mathematics, Swansea University, 
Swansea University Bay Campus,
Fabian Way,
Swansea,
  Swansea SA1 8EN, U.K.}

\email{Bernard.Rybolowicz@swansea.ac.uk}



\begin{abstract}
Two observations in support of the thesis that trusses are inherent in ring theory are made. First, it is shown that every equivalence class of a congruence relation on a ring or, equivalently, any element of the quotient of a ring $R$ by an ideal $I$ is a paragon in the truss $\mathrm{T}(R)$ associated to $R$. 
Second, an extension of a truss by a one-sided module is described. Even if the extended truss is associated to a ring, the resulting object is a truss, never a ring, unless the module is trivial. On the other hand, if the extended truss is associated to a brace, the resulting truss is also associated to a brace, irrespective of the module used.
\end{abstract}
\date\today
\maketitle

\section{Introduction}
The aim of this paper is to demonstrate that points of a quotient ring, understood as equivalence classes of the relation defined by an ideal, and extensions of rings by one-sided modules have a rich algebraic structure which leads one from the realm of rings to that of trusses. It has been observed first by Pr\"ufer \cite{Pru:the} in the abelian case and then by Baer \cite{Bae:ein} and Su\v skevi\v c \cite{Sus:the} in general that every group can be understood as a herd ({\em Schar}) or heap ({\em Gruda}), i.e.\ a set with an associative ternary operation that satisfies Mal'cev identities. In contrast to groups no choice of a specific element is made in a heap, indeed any point of a heap can play the role of the neutral element of the associated group, known as a {\em retract} of the heap.  This allows one to view an empty set as a heap (which is one of the indications that the notion of a heap generalises that of a group) and to understand heaps as an affine version of groups. Recently, a similarly affine version (or generalisation) of a ring has been introduced in \cite{Brz:tru} and termed a {\em truss}. The original motivation for trusses comes from exciting progress in the study of sets with two interacting group structures introduced by Rump  \cite{Rum:bra}, and called {\em braces} (see also \cite{CedJes:bra}, \cite{GuaVen:ske}), but the simplicity of the definition of a truss and both similarity and unexpected dissimilarity between rings and trusses, makes the latter an interesting object of studies on their own \cite{Brz:par}, \cite{BrzRyb:mod}. In this paper, we show two instances in which trusses enter into most natural ring-theoretic considerations and, indeed, appear to be indispensable for a full understanding of these constructions.

More specifically, Certaine observed in \cite{Cer:ter} that every coset in a group while not a subgroup in general is a sub-heap of the heap associated with the group. Thus, in order to treat all cosets on equal footing, one needs to depart the category of groups for the category of heaps. In the present paper we show that given a ring $R$, a left module $M$ and a submodule $N$ of $M$, the cosets of $N$ are (induced) modules of the truss $\tT(R)$ associated to $R$. In particular, the points of the quotient ring $R/I$, i.e.\ the cosets of the ideal $I$ in $R$ are induced sub-bimodules  or {\em paragons} of $\tT(R)$. 
Secondly, we show that trusses are deeply embedded in the group and ring theory; ring extensions by one-sided modules, including extensions of the endomorphism ring of an abelian group by this group, typically yield trusses rather than rings. More precisely, given a ring $R$ and a left $R$-module $M$, we define an associative multiplication on $R\oplus M$, which, while it does not distribute over the addition, distributes over the associated ternary heap operation;  thus $R\oplus M$ is a truss. 
We present this construction more generally for trusses thus making it applicable to braces too. 

The paper is organised as follows. Section~\ref{sec.prem} contains preliminary information on heaps, trusses their paragons and modules, including special cases of trusses associated to rings and braces. Section~\ref{sec.eqcl} is devoted to the analysis of the algebraic structure of a congruence equivalence classes in rings (and trusses). In the first main theorem, Theorem~\ref{thm.cong} we show that there is one-to-one correspondence between congruence equivalence classes in modules over trusses with induced submodules. In particular every such class in a module over a ring is a module over the truss associated to this ring, not a module over the ring (unless the class contains zero). We demonstrate how to translate these classes to modules over rings. Furthermore, we conclude that elements of a quotient ring $R/I$ are paragons in the  truss $\tT(R)$ associated to $R$. We then study rings for which the set of all units forms a paragon, in particular we find that such rings with the quotient by the paragon of units being $\ZZ_2$ are fully characterised by the property that either an element is a unit or else the difference between this element and the identity is a unit; see Theorem~\ref{thm.z2}. The section is concluded with an observation that if the set of units in a truss $T$ is a sub-heap, then it forms a brace. We apply this observation to commutative brace-type truss structures on the additive group $\ZZ$ classified in \cite{Brz:par}. By taking suitable quotients of one of these trusses we recover two-sided braces with additive structure of $C_{2^{k+1}}$ and multiplicative structure of $C_{2}\oplus C_{2^k}$, i.e.\ all abelian cyclic braces with socles of order 2 \cite[Proposition~4]{Rum:cla}.

In Section~\ref{sec.ext} we discuss and analyse extensions of trusses $T$ to trusses $T[M;e]$ by one-sided modules $M$ and relative to $e\in M$. These are constructed on the product heap $T\times M$. It is shown that the ring-type truss is only trivially extended to a ring type truss, i.e.\ $T[M;e]$ is ring-type if and only if $T$ is ring-type and $M=\{e\}$. Nevertheless even if $M$ is a non-trivial module, the resulting truss can be extended to a ring, following the procedure of adjoining the zero element described in \cite{BrzRyb:mod}. The main result of this section,  Theorem~\ref{thm.mod} gathers  properties of such extensions. In particular it is observed that although the extension is defined relatively to an element of a module, the trusses obtained for different elements are mutually isomorphic. The module $M$ has a natural action of $T[M;e]$ and  can be embedded into $T[M;e]$ as a paragon so that the quotient truss  is isomorphic to $T$. In a similar way $T$ can be embedded in $T[M;e]$ as a left paragon; the resulting quotient $T[M;e]$-module  returns $M$. Put together these mean that we are dealing with a split extension of trusses. Finally $T[M;e]$ is unital if and only if $T$ is unital, and units of $T[M;e]$ are formed of the product of units of $T$ and the whole of $M$. Thus, in particular, if $T$ corresponds to a brace, so does $T[M;e]$. We use this last observation to construct an extension of an abelian brace with additive structure of $C_4$ and multiplicative structure of $C_2\oplus C_2$ to the non-abelian brace with additive structure $C_4\oplus C_4$ and multiplicative structure $C_2\oplus D_8$. We note that this construction can be also performed for left (one-sided) trusses. In the special case in which $B$ is a left brace acting on an abelian group $M$ by automorphism of $M$, then the extension of the corresponding left truss $\tT(B)$ by the $M$ at 0, i.e.\ the left-truss $\tT(B)[\tT(M);0]$ corresponds to the special case of extensions of $B$ introduced in \cite[Theorem~3.3]{Bac:ext}.

\section{Preliminaries}\label{sec.prem}
\subsection{Trusses}\label{sec.truss} 
Following \cite{Brz:tru} a \textbf{truss} is a set $T$ together with a ternary operation $[-,-,-]:T\times T\times T\to T$ and multiplication $\cdot:T\times T\to T$, denoted by juxtaposition, such that for all $a,b,c,d,e\in T$,
\begin{zlist}
\item $[[a,b,c],d,e]=[a,b,[c,d,e]]$ (associativity),
\item $[b,b,a]=a=[a,b,b]$ (Mal'cev identities),
\item $[a,b,c]=[c,b,a]$ (commutativity),
\item $a\cdot[b,c,d]=[a\cdot b,a\cdot c,a\cdot d]$ (left distributivity),
\item$ [b,c,d]\cdot a=[b\cdot a,c\cdot a,d \cdot a]$ (right distributivity).
\end{zlist}
We say that $T$ is \textbf{commutative} if, for all $a,b\in T$, $ab=ba.$  A truss is said to be \textbf{unital} if it has the identity (often denoted by $1$) with respect to the multiplication. If only conditions (1)--(4) are satisfied one refers to $T$ as a \textbf{left truss} and in case all but (4) hold $T$ is termed a \textbf{right truss}.\footnote{Further possibilities include dropping condition (3) thus yielding skew trusses, but we will not be dealing with these notions in here.}

The conditions (1)--(2) mean that $(T,[-,-,-])$ is a heap and (1)--(3) that it is an abelian heap (see e.g.\ \cite{Pru:the}, \cite{Bae:ein}, \cite{Sus:the}). A \textbf{sub-heap} is a subset of a heap closed under the heap operation. Every (abelian) group has the unique associated (abelian) heap structure with the ternary operation $[a,b,c]=ab^{-1}c$. Conversely, fixing the middle term in the heap operation,  one obtains a group operation $[-,e,-]$, for which $e$ is the neutral element. This group is known as a \textbf{retract} of the heap (or $e$-retract, if the fixed neutral element is to be specified). All retracts of the same heap are mutually isomorphic with the isomorphism from the $e$-retract to the $e'$-retract given by
$$
\tau_e^{e'}(a) = [a,e,e'], \qquad \mbox{for all $a\in T$}.
$$

Every ring $R$ can be viewed as a truss, by replacing addition by the induced heap operation $[a,b,c]=a-b+c$, for all $a,b,c\in R$. This truss is denoted by $\tT(R)$  and it is an example of a \textbf{ring-type truss}, that is a truss $T$ with an \textbf{absorber}, i.e. an element $0\in T$ such that, for all $t\in T$, $0t=t0=0$. A truss can possess at most one absorber. 

A \textbf{truss morphism} is a function between trusses $\varphi:T\to T'$ which preservers operations, i.e., for all $a,b,c\in T$,

$$
\varphi([a,b,c])=[\varphi(a),\varphi(b),\varphi(c)]\quad \& \quad \varphi(ab)=\varphi(a)\varphi(b).
$$
In the case of rings, every morphism between two rings $\varphi:R\to R'$ is also a truss morphism between associated trusses; $\tT(\varphi):\tT(R)\to \tT(R')$ is given by $\tT(\varphi)(r)=\varphi(r)$, for all $r\in R$.  As presented notation suggests $\tT$ is a functor between the category of rings and that of trusses.

In rings one needs the notion of an ideal to create quotients, in trusses  the required structure is that of a paragon. A \textbf{paragon} is a sub-heap $P$ of a truss $T$ for which there exists $q\in P$, such that, for all $x\in T$ and $p\in P$,
$$\lambda^{q}(x,p):=[xp,xq,q]\in P\quad \& \quad \varrho^{q}(p,x):=[px,qx,q]\in P.$$

If $P$ is a paragon in a truss $T$, then the canonical map $\pi:T\to T/P$ is an epimorphism of trusses (see \cite[ Proposition~3.22]{Brz:par}). Here $T/P$ is the set of equivalence classes of the sub-heap relation:
$$
x\sim_Py \iff \exists p\in P,\;\; [x,y,p]\in P.
$$
The inverse image of any element $e$ in the codomain of a truss homomorphism $\varphi: T \lra T'$ is a paragon in $T$ (called an $e$-kernel of $\varphi$). Furthermore, $\varphi^{-1}(e)$ is a sub-truss of $T$ if and only if $e$ is an idempotent. 
A sub-heap $P$ is a left paragon if it is closed under $\lambda^q$ and a right paragon if it is closed under $\varrho^q$, for some (and hence for all) elements $q\in P$. Although paragons describe fully congruences in trusses, i.e.\ the sub-heap relation $\sim_P$ is a congruence in a truss $T$ if and only if $P$ is a paragon, one can still consider analogs of ideals in trusses, see \cite{Brz:par}. An \textbf{ideal} of a truss $T$ is a sub-heap $J\subseteq T$ such that for all $t\in T$ and $x\in J,$ $tx\in J$ and $xt\in J.$ Obviously every ideal is a paragon. Furthermore, $T/J$ is ring-type with the absorber $J$.

\subsection{Modules over trusses}\label{sec.mod} Let $T$ be a truss. A \textbf{left} $T$\textbf{-module} is an abelian heap $M$ together with an action $\lambda_M:T\times M\to M$ of $T$ on $M,$ written on elements as $t\cdot m,$ that is associative and distributes over the heap operation, i.e., for all $t,t',t''\in T$ and $m,m',m''\in M,$
$$
\begin{aligned}
t\cdot( t'\cdot m)& =(t t')\cdot m, \\
[t,t',t'']\cdot m &=[t\cdot m,t'\cdot m,t''\cdot m],\\
t\cdot[m,m',m''] &=[t\cdot m,t\cdot m',t\cdot m''].
\end{aligned}
$$
We say that a module $M$ is \textbf{unital} if $T$ has an identity and $1\cdot m=m$ for all $m\in M.$ In a symmetric way right modules are defined as left modules over the opposite truss, hence whatever is said about left modules can be stated about right modules too. Thus 
we will mainly consider left modules and further we will use term `module' for `left module'.
A \textbf{module morphism} is a morphism of heaps that preserves the actions.

An element $e$ of a left $T$-module $M$ is called an \textbf{absorber} if, for all $t\in T$, $t\cdot e=e.$ 
Since morphisms of modules preserve actions they also preserve absorbers. It is worth mentioning here that a module need not have absorbers and that, contrary to the case of trusses,
 even in a bimodule, 
a (two-sided) absorber is not unique; for example if $M$ is an abelian heap then $M$ can be treated as a bimodule over an arbitrary truss $T$ with operations given, for all $t\in T$ and for all $m\in M$, by $t\cdot m=m=m\cdot t$.
A module $M$ over a ring $R$ gives rise to the module $\tT(M)$ over the truss $\tT(R)$ with unchanged action and the heap operation given by addition.  
Every $R$-module homomorphism $\varphi$ gives rise to the $\tT(R)$-module homomorphism, $\tT(\varphi)=\varphi$. Therefore, $\tT$ is a functor from the category of modules over a ring $R$ to the category of modules over the associated truss $\tT(R)$.
It has been observed in \cite{BrzRyb:mod} that if $R$ is a ring, $M$ and $N$ are $\tT(R)$-modules with unique absorbers, say, $e$  in $M$ and  $e'$ in  $N$, and $\varphi: M\to N$ is a $\tT(R)$-module morphism, then since $\varphi$ preserves absorbers,  $\varphi$ is a morphism between $R$-modules $(M,[-,e,-])$ and $(N,[-,e'-])$.

Similarly to the case of modules over rings one can define a \textbf{submodule} of a module over a truss as a sub-heap closed under the action. Although it is a well defined subobject and it is possible to define quotient modules in this way, the resulting quotient module will have at least one absorber, so a more general  approach needs to be taken. Following \cite[Section 4.9]{Brz:par}, first observe that any element $e$ of a $T$-module $M$ induces an action, for all $t\in T$ and $m\in M$, 
\begin{equation}\label{act.ind}
t\cdot_e m = [t\cdot m, t\cdot e,e].
\end{equation}
The heap $M$ together with the induced action $\cdot_e$ is a $T$-module with absorber $e$. 
Next, an \textbf{induced submodule} $N$ of a module $M$ is a sub-heap closed under the induced action, i.e.\ for all $e,e'\in N$ and all $t\in T$ $t\cdot_{e}e':=[te',te,e]\in N.$ 
If $N$ is an induced submodule of $M$, then the canonical map $\pi_{N}: M\to M/N$ is an epimorphism; conversely, if $N$ is a sub-heap of $M$ such that $M/N$ is a $T$-module and the canonical surjection is a module epimorphism, then $N$ is an induced  submodule of $M$ (see \cite[Proposition 4.32(2)]{Brz:par}).

If  $\varphi:M\to M'$ is a module morphism, then, for all $m'\in M'$, $\varphi^{-1}(m')$ is an induced submodule of $M$. 

\subsection{Braces}\label{sec.braces}
 Following \cite{Rum:bra}, \cite{CedJes:bra}, \cite{GuaVen:ske} a \textbf{two-sided brace} is a set $B$ with two group operations: $+:B\times B\to B,$ which is commutative, and $\cdot:B\times B\to B,\ (a,b)\mapsto ab$ such that for all $a,b,c \in B,$ 
\begin{equation}\label{dist.brace}
a (b+c)=ab-a+a c\quad \text{\ \&\ }\quad (b+c) a=ba-a+c a.
\end{equation}
A set $B$ with two group operations satisfying the first of conditions \eqref{dist.brace} is known as a \textbf{left brace}; if only the second of \eqref{dist.brace} is satisfied one refers to $B$ as to a \textbf{right brace}.  It is possible to define a brace starting with a truss. Let $T$ be a truss that is a group with respect to the multiplication.  Then  $B=(T,+,\cdot),$ where $-+-:=[-,1,-]$ and $\cdot$ is the multiplication in $T$, is a two-sided brace (see \cite[Corollary~3.10 ]{Brz:par}). Note that $1$ is a neutral element for both $\cdot$ and $+$. In the converse direction, to any two-sided  brace $B$ a truss can be associated by equipping $B$ with the ternary heap operation $[a,b,c]:=a-b+c$,  for all $a,b,c\in B$. This truss will be denoted by $\tT(B).$ In case $B$ is a left brace, $\tT(B)$ is a left truss and if $B$ is a right brace, then $\tT(B)$ is a right truss. 
It might be  worth mentioning that $\tT(B)$ is the main example that motivated the definition of a truss. 

 Following \cite[Definition~3]{CedJes:bra} and the discussion in there an \textbf{ideal} of a  (left, right, two-sided) brace $B$ is a normal subgroup $S$ of $(B,\cdot)$ such that for all $b\in B$ and $s\in S,$
$$
bs-s\in S \quad \mbox{equivalently} \quad sb-s \in S .
$$ 
The foregoing condition can be written in the associated  truss $\tT(B)$ as the closeness under $\lambda^{1}$  or,  equivalently, under $\varrho^1$, since, for all $b\in \tT(B)$ and $s\in S,$
$$
\begin{aligned}
\lambda^{1}(b,s)&=[bs,s,1]=bs-s+1=bs-s\in S,\\
\varrho^1(s,b)&= [sb,s,1] = sb-s+1 = sb-s \in S
\end{aligned}
$$
If $S$ is an ideal, then the canonical map $\pi:B\to B/S$ is an epimorphism of left braces.

  \section{On the algebraic structure of an equivalence class}\label{sec.eqcl}

 The aim of this section is to grasp the internal algebraic structure of points in a module over a ring or, more generally, a module over a truss. Every module (over a ring or, more generally, truss) is a quotient module, for example of a free module. Seen in that way, points in any module have internal  structure of an equivalence class of a particular relation (congruence). If an $R$-module $M$ is a quotient of an $R$-module $N$, then the zero of  $M$ is a submodule of $N$. What is the structure of all the remaining points? As equivalence classes of a congruence they are sets which form a partition of $N$ compatible with the $R$-action. Therefore one might expect that they will carry additional algebraic structure which reflects addition and $R$-action in $N$. Inspired by Certaine's observation \cite{Cer:ter} that a subset $S$ of a group $G$ is a coset for some subgroup $G'$ of $G$ if and only if $S$ is a sub-heap of the unique heap associated to $G$, we aim in this section to reveal this structure. It turns out that one needs to view rings as trusses to achieve this aim. 
 
Let $T$ be a truss and let $M$ be a $T$-module. Given a congruence 
$\sim$  in  $M$, let  $\pi_{\sim}:M\to M/\sim$ denote the canonical epimorphism that sends elements of $m$ to their equivalence classes. Note that $a\sim b$ if and only if $\pi_{\sim}(a)=\pi_{\sim}(b)$. Therefore $\sim$ is the  kernel relation $\Ker(\pi_{\sim})$. Since every kernel relation is a congruence (for any algebra), one can identify congruences in $M$ with kernel relations of morphisms of $T$-modules. 

\begin{theorem}\label{thm.cong}
Let $T$ be a truss, $M$ a $T$-module and $N$ a subset of $M$. Then the following statements are equivalent:
\begin{zlist}
\item $N$ is an induced submodule of $M$.
\item $N$ is a sub-heap of $M$ and  there is a congruence $\sim$ on $M$ such that $\pi_N = \pi_\sim$.
\item $N$ is an equivalence class of a congruence on $M$.
\end{zlist}
\end{theorem}
\begin{proof}
(1) $\implies$ (2): If $N$ is an induced submodule of $M$, then the sub-heap relation $\sim_N$ is a congruence  on $M$ by (the proof of) \cite[Proposition~4.32]{Brz:par}.

(2) $\implies$ (3): Every induced submodule is a sub-heap and every sub-heap is an equivalence class for its sub-heap relation, see \cite[Proposition~2.10]{Brz:par}.

(3) $\implies$ (1):  Assume that $N$ is an equivalence class for a congruence on $M$, say $\sim$. In particular $\sim$ is a congruence for the  heap structure on $M$, hence $N$ is a sub-heap by \cite[Theorem~1]{Cer:ter}. Furthermore, for all $t\in T$ and $n,n'\in N$, $t\cdot n \sim t\cdot n'$, i.e.\ $\pi_\sim(t\cdot n) = \pi_\sim( t\cdot n')$. Set $m = [t\cdot n, t\cdot n', n']$. Then
$$
\begin{aligned}
\pi_\sim(m) &= \pi_\sim\left([t\cdot n, t\cdot n', n']\right) \\
&= [\pi_\sim(t\cdot n), \pi_\sim( t\cdot n'),\pi_\sim(n')]
 = [\pi_\sim(t\cdot n), \pi_\sim( t\cdot n),\pi_\sim(n')].
\end{aligned}
$$
Hence, $\pi_\sim(m)=\pi_\sim(n')$, i.e.\ $m\in N$.  Thus $N$ is an induced submodule.
\end{proof}

Theorem~\ref{thm.cong} can be applied to trusses associates to rings, thus yielding
\begin{corollary}\label{cor.cong}
If $M$ is a module over a ring $R,$ then $N\subseteq M$ is an equivalence class for a congruence $\sim$ in $M$ if and only if $N$ is an induced submodule of $\tT(M).$
\end{corollary}
\begin{proof}
Suffices it to observe that an equivalence relation  is a congruence on $M$ as an $R$-module if and only if it is a congruence on $M$ as a $T(R)$-module and then apply Theorem~\ref{thm.cong}.
\end{proof}

Since every congruence relation of $R$-modules arises as the quotient by a submodule, Corollary~\ref{cor.cong} gives interpretation of points of quotients of $R$-modules $M/N$ as induced submodules of $\tT(M)$ ($M$ viewed as a module of the associated truss $\tT(R)$). Furthermore, it provides one with the procedure of calculating the quotient of an $R$-module $M$ by the equivalence class of any point $m$ of $M$: one simply needs to interpret $M$ as a heap and then take the quotient by the class of $m$ which is a sub-heap of $M$. In the same vein one obtains the following interpretation of points of a quotient and hence of any ring.

\begin{corollary}\label{cor.cong.ring}
Let $R$ be a ring and $I$ an ideal in $R$. Then every point of the quotient ring $R/I$ is a paragon in the associated truss $\tT(R)$.
\end{corollary}

 The following proposition explains how to recover the submodule from any equivalence class in the quotient $R$-module $M/N$. 
  \begin{proposition}\label{prop.para.shift}
Let $M$ be a module over a truss $T$ and let $N$ be a left (resp.\ right) induced submodule of $M$. For all $e\in N$ and $m\in M$, define the sub-heap
\begin{equation}\label{para.shift}
N_{e}^{m} = \tau_{e}^m(N) = \{[n,e,m]\; |\; n\in N\}.
\end{equation}
Then:
\begin{zlist}
\item $N_{e}^{m}$ is a left (resp.\ right) induced submodule.
\item $N_{e}^{m}$ is a left (resp.\ right) submodule if and only if, for all $t\in T$, 
$$
t\cdot m\in N_{e}^{m}, \qquad \mbox{(resp.\ $m\cdot t\in N_{e}^{m}$)}.
$$
\item If $m\not\in N$, then $N\cap N_{e}^{m} = \emptyset$.
\end{zlist}
\end{proposition}
\begin{proof}
(1) First let us note that, as a consequence of Mal'cev idenities, $m= \tau_e^m(e) \in N_{e}^{m}$. For all $t\in T$,
$$
\begin{aligned}
~[t\cdot[n,e,m],t\cdot m,m] &= [[t\cdot n,t\cdot e,t\cdot m],t\cdot m,m]\\
& = [t\cdot n,t\cdot e,m] = [[t\cdot n,t\cdot e,e],e,m] = \tau_e^m([t\cdot n,t\cdot e,e]),
\end{aligned}
$$
by the (left) distributive law, associativity and Mal'cev identities. Since $N$ is a left induced module, $[t\cdot n,t\cdot e,e]\in N$, and so 
$$
[t\cdot [n,e,m],t\cdot m,m] \in \tau_e^m (N) = N_e^m,
$$
as required. The case of a right induced module is dealt with symmetrically.

(2) Obviously, if $N_e^m$ is a submodule and since $m\in N_e^m$,  $t\cdot m\in N_e^m$. Conversely, if $t\cdot m\in N_e^m$, let $n_m\in N$ be such that
$$
t\cdot m = \tau_e^m(n_m) = [n_m, e,m].
$$
Then, for all $n\in N$,
$$
\begin{aligned}
t\cdot [n,e,m] & = [t\cdot n,t\cdot e,t\cdot m] = [t\cdot n, t\cdot e, [n_m, e,m]]\\
& = [[[t\cdot n,t\cdot e,e],e,n_m], e,m]
= \tau_e^m([[t\cdot n,t\cdot e,e],e,n_m]),
\end{aligned}
$$
by the (left) distributive law, associativity and the Mal'cev identities. Since $N$ is an induced submodule, $[t\cdot n,t\cdot e,e]\in N$. Consequently, $[[t\cdot n,t\cdot e,e],e,n_m]\in N$, and therefore $t\cdot \tau_e^m(n) \in \tau_e^m(N) = N_e^m$, for all $n\in N$ and $t\in T$, as required.

(3) Suppose that $n\in N$ is an element of $N_{e}^{m}$, so that there is $n'\in N$ such that $n=[n',e,m]$. By the associativity of $[-,-,-]$ and the Mal'cev identities, $m = [n',e,n]\in N$, which contradicts the assumption that $m\not\in N$.
\end{proof}

Since $N_{e}^m$ is the image of $N$ under the heap isomorphism $\tau_e^m$,   $N_{e}^m\cong N$ as  heaps.  By the foregoing proposition this isomorphism is an isomorphism of induced modules, therefore $M/N_{e}^{m}\cong M/N.$ In the case of a module $M$ over a ring $R$, an $R$-submodule is obtained from the point $N$ of the quotient module by choosing $m=0$, the zero of $M$.
  
  \begin{example}\label{ex.para.Z}
View the ring $\ZZ$ as a truss with the heap operation $[k,l,m]_+= k-l+m$ and the usual multiplication of integers. For any $n\in \NN$, consider the ideal $n\ZZ$. Then, for all $m\in \ZZ$,
$$
(n\ZZ)_0^m= \tau_0^m(n\ZZ) = \{kn+m\; |\; k\in \ZZ\} = \{kn +r\; |\; k\in \ZZ\},
$$
where $r$ is a remainder of the division of $m$ by $n$, is a paragon. In particular if $n$ does not divide $m$, or, equivalently, $0< m <n$ then $(n\ZZ)_0^m$ is not an ideal. One easily checks that
$$
\tT(\ZZ)/(n\ZZ)_0^m\cong\tT( \ZZ/n\ZZ).
$$
For example, $(2\ZZ)_0^1$ is the set of all odd integers but one can translate it to an ideal by taking $((2\ZZ)_0^1)_{1}^0=2\ZZ.$ In spite of the fact that $(2\ZZ)_0^1$ is not an ideal and that it contains the identity of $\ZZ$ it is a  paragon different from $\ZZ$, with corresponding quotient being a non-trivial ring. 
\end{example}

\begin{definition}
Let $T$ be a truss. A paragon $P\subseteq T$ is said to be \textbf{normal} if, for all $t\in T,$ $tP=Pt$.
\end{definition}

\begin{proposition}
Let $B$ be a two-sided brace. Then:
\begin{zlist}
\item $S\subseteq B$ is an ideal if and only if $S$ is a normal paragon in $\tT(B)$ and $1\in S.$ \label{prop:brace1}
\item Let $I$ be an ideal in $B,$ then $S\in B/I$ if and only if $S$ is a normal paragon in $\tT(B).$ \label{prop:brace2}
\end{zlist}
\end{proposition}
\begin{proof}(1) First, if $S$ is an ideal, then $S$ is a normal paragon in $\tT(B)$ and $1\in S$ (see \cite[Definition~3]{CedJes:bra} or the discussion at the end of Section~\ref{sec.braces}). In the converse direction, assume that $S$ is a normal paragon in $\tT(B)$ and $1\in S.$ It is enough to show that $(S,\cdot)$ is a normal subgroup since $S$ is closed under $\lambda^{1}$ (again see \cite [Definition~3]{CedJes:bra} or the discussion at the end of Section~\ref{sec.braces}). Let $a,b\in S$, then
$$
S\ni [a,\lambda^{1}(ab^{-1},b),1]=[a,[ab^{-1}b,ab^{-1},1],1]=[a,[a,ab^{-1},1],1]=[a,a,ab^{-1}]=ab^{-1},
$$ 
by the commutativity of the heap operation, associative laws and Mal'cev identities. Therefore $S$ is a subgroup  of $B$.  Finally, the normality of the paragon implies the normality of the group.

(2) Let us assume that $S\in B/I$ for an ideal $I.$ Then $S=I+a=I-1+a=I_{1}^{a}$ for some $a\in B.$ As paragons are induced modules in trusses, by statement \eqref{prop:brace1}  and by Proposition~\ref{prop.para.shift}(1), $S$ is a normal paragon. Conversely, assume that $S$ is a normal paragon in $\tT(B)$. Then by Proposition~\ref{prop.para.shift}(1), for any $a\in S,$ $S_{a}^{1}$ 
is a normal paragon and $1\in S_{a}^{1}.$ Therefore, by assertion (1)  $S_{a}^{1}$ 
is an ideal in $B$ and $S=(S _{a}^{1})_{1}^{a}=S_{a}^{1}-1+a=S_{a}^{1}+a\in B/S_{a}^{1}.$
\end{proof}

\begin{example}
Let $B$ be a (left, right or two-sided) brace. Following \cite[Section~4]{CedJes:bra}, the  \textbf{socle}  $\mathrm{Soc}(B)$ of $B$ is defined as
$$
\begin{aligned}
\mathrm{Soc}(B)&=\{a\in B\;|\; ab=a+b,\ b\in B\}.
\end{aligned}
$$ 
The socle  is an ideal of a brace that is non-trivial, i.e.\ different from $\{1\}$ if $B$ is non-trivial and finite (see \cite[Proposition~3]{CedJes:bra} or \cite{Rum:bra}).  
For all $c\in B$,  
$$
c+\mathrm{Soc}(B)=\{c+a\;|\;a\in \mathrm{Soc}(B)\}
$$ 
is an equivalence class of a congruence in $B$ and  hence it is a paragon in $\tT(B)$. Indeed,   $c+\mathrm{Soc}(B)$ is a sub-heap of $\tT(B)$, since for all $a,a',a''\in \mathrm{Soc}(B)$,
$$
[c+a,c+a',c+a'']=c+a-c-a'+c+a''=c+(a-a'+a'')\in c+\mathrm{Soc}(B).
$$
Furthermore, for all $b\in B$,
$$\varrho^{c}(c+a,b)=[c,cb,(c+a)b]=c-cb+cb-b+ab=c+ab-b\in c+\mathrm{Soc}(B),
$$
and
$$
\lambda^{c}(b,c+a)=[b(c+a),bc,c]=bc-b+ba-bc+c=c+(bab^{-1})b-b\in c+\mathrm{Soc}(B),
$$
by  the fact that $\mathrm{Soc}(B)$ is a normal subgroup of $(B,\cdot)$.
Therefore, $c+\mathrm{Soc}(B)$ is closed under  right and left induced actions,  and hence it is a paragon in $\tT(B)$.

\end{example}

\begin{example} Let $RG$ be a group ring for an arbitrary ring $R$ and an arbitrary group $G.$ Let us observe that, for all $r\in R$, the  sets 
$$
A_{r}:=\{\sum_{g\in G}r_{g}g\;|\;\sum_{g\in G}r_{g}=r\}
$$
are paragons in $\tT(RG)$ as inverse images of $r$ under the ring (and hence truss) homomorphism 
$$
\pi: RG\lra R, \qquad  \sum_{g\in G}r_{g}g\lto  \sum_{g\in G}r_{g}.
$$
In particular, each 
 $A_{r}$ is a sub-truss only if $r$ is an idempotent. 
It can be easily checked (or deduced from the fact that $\pi$ is an epimorphism combined with the first isomorphism theorem for algebras) that, for all $r\in R,$ 
$$
\mathrm{T}(RG)/A_{r}\cong \mathrm{T}(R).
$$
\end{example}

In the remainder of this section we study when the group of units $\uU(R)$ of a ring $R$ is an equivalence class for a congruence in $R$ or, equivalently, when $\uU(R)$ forms a paragon in $\tT(R)$.

 \begin{lemma}\label{lem.g.ring}
Let $RG$ be a group ring. If $\uU(RG)$ is a paragon in $\tT(RG),$ then $\uU(R)$ is a paragon in $\tT(R).$
\end{lemma}
 \begin{proof}
If $e$ is  the neutral element of  $G,$ then it is easy to check that $re\in\uU(RG)$ if and only if $r\in \uU(R).$ In view of this observation and from the fact that $\uU(RG)$ is a paragon, if $r,r',r''\in \uU(R),$ then $[re,r'e,r''e]=[r,r',r'']e\in\uU(RG),$ which implies that $[r,r',r'']\in \uU(R)$. Hence $\uU(R)$ is a sub-heap of $\tT(R)$. Furthermore, for all $r\in R$, and  $r',r''\in \uU(R)$,
$$
\uU(RG)\ni \lambda^{r'e}(re,r''e)=[rr''e,rr'e,r'e]=[rr'',rr',r']e=\lambda^{r'}(r,r'')e,
$$ 
which implies that $ \lambda^{r'}(r,r'')\in\uU(R).$  Similarly, $\varrho^{r'}(r'',r)\in \uU(R)$. Therefore, $\uU(R)$ is a paragon in $\tT(R).$
\end{proof}

Now one can ask when units are an equivalence class? Although at this stage  we  are not able to provide sufficient conditions we still can provide some necessary conditions.
 \begin{proposition}\label{prop.units}
 Let $R$ be a ring and assume that $\uU(R)$ is a paragon in  $\tT(R)$. Then
 \begin{zlist}
 \item  For all $a,b\in\uU(R),$ $a-b\not\in \uU(R).$ 
\item Odd multiples of units in $R$ are units while even ones are not.
\item The quotient truss $\tT(R)/\uU(R)$ corresponds to a ring of characteristic 2.
 \end{zlist}
 \end{proposition}
 \begin{proof}
 (1) If $\uU(R)$ is a paragon in $\tT(R)$ or, equivalently by Corollary~\ref{cor.cong.ring} an equivalence class for a congruence in $R$, then its translate $\uU(R)_{b}^{0}$  is an ideal in $R$ (see Proposition~\ref{prop.para.shift}). Explicitly,  $\uU(R)_{b}^{0}$ consists of elements of the form $a-b$ where $a\in\uU(R)$. Since $0\not\in \uU(R)$, $\uU(R)_{b}^{0}\cap\uU(R)=\emptyset$ by Proposition~\ref{prop.para.shift}(3) and hence $a-b\not\in \uU(R)$. \smallskip

(2) Note that for any $a\in \uU(R)$, $-a\in \uU(R)$. We first prove by induction that  $(2n+1)a\in \uU(R)$, for all $n\in \NN$. If $n=0$, then the statement is obvious.
 Now assume that $(2n+1)a\in \uU(R),$ then 
 $$
 (2n+3)a = [(2n+1)a,-a,a]\in\uU(R),
 $$
 so for all positive odd numbers and thus also the negative ones the assertion is true. Since any even multiple is a difference of two odd multiples, the second assertion follows by (1).\smallskip
 
 (3) The equivalence class of $1\in R$,  $\overline{1}=U(R)$ is the identity in the quotient truss $\tT(R)/\uU(R)$. Since both $1$ and $-1$ are units in $R$, $\overline{-1} = \overline{1}$, which implies that $\overline{1} +\overline{1} = 0$ in the ring corresponding to $\tT(R)/\uU(R)$.
 \end{proof}
 
 \begin{example}\label{ex.z4}
 Let us consider ring $\ZZ_{4}$ and its associated  truss $\tT(\ZZ_{4})$. It is easy to check that the set $\uU(\ZZ_{4})=\{1,3\}$ is a paragon in $\tT(\ZZ_{4}).$ Therefore, $\uU(\ZZ_{4})$ is an element in the quotient of $\ZZ_{4}$ by the ideal $\uU(\ZZ_{4})_{1}^{0}=\{1,3\}_{1}^{0}=\{0,2\}$ i.e. $\ZZ_{2}.$
 \end{example}
 
Example~\ref{ex.z4} shows that rings in which units form equivalence classes of congruences exist. The following theorem classifies all rings in which units form a paragon and the quotient truss corresponds to the ring $\ZZ_2$.
\begin{theorem}\label{thm.z2}
For a ring $R$, the following statements are equivalent:
\begin{zlist}
\item The units $\uU(R)$ form a paragon in $\tT(R)$ and $\tT(R)/\uU(R)\cong\tT(\ZZ_{2})$.
\item For all $r\in R$, either $r\in \uU(R)$ or $1-r\in \uU(R)$.
\end{zlist}
\end{theorem}
\begin{proof}
Assume first that the statement (1) holds. Since $\tT(R)/\uU(R)\cong\tT(\ZZ_{2})$ there are two disjoint paragons in $\tT(R)$, $\uU(R)$ and $J$, covering the whole of $\tT(R)$. Since $0\not\in \uU(R)$, $0\in J$ and hence $J$ is an ideal. Since $1\in \uU(R)$,
$
\uU(R) = \tau_0^{1}(J),
$
by Proposition~\ref{prop.para.shift}. Take any $r\in R$, then either $r\in \uU(R)$ or $r\in J$, in which case also $-r\in J$ (as $J$ is an ideal) and
$
1-r = \tau_0^{1}(-r) \in \uU(R).
$
Hence the statement (2) holds.

In the converse direction, the assumption (2) means in particular that $R$ is a local ring, i.e.\ the set of non-units, say $J$, is an ideal in $R$. Note that if $u\in \uU(R)$ and $r\in J$, then $u+r \in \uU(R)$, for should $u+r$ not be a unit, then $1-u-r$ would be a unit, hence not an element of $J$, which would contradict the fact that $J$ is an ideal, as $1-u-r\in J$. 

Take any $u,v,w\in \uU(R)$. Then
$$
[u,v,w] = u-v+w = (u-1) +(1-v) +w\in \uU(R),
$$
by the preceding discussion, as $u-1, 1-v\in J$ and $J$ is an ideal. Hence $\uU(R)$ is a sub-heap of $\tT(R)$. Next, take any $r\in R$ and $u\in \uU(R)$. Then, by the same token
$$
[ru,r,1] = r(u-1) +1 \in \uU(R).
$$
Therefore, $\uU(R)$ is a paragon in $\tT(R)$.

Finally, take any $r \in R$. If $r\in \uU(R)$, then its class $\bar{r} \in \tT(R)/\uU(R)$ is equal to $\uU(R) = \bar{1}$. If $1-r\in \uU(R)$, then also $r-1\in \uU(R)$, and 
$$
r = [r-1,-1,0] = \tau_{-1}^0(r-1)\in \tau_{-1}^0(\uU(R)) ,
$$
so that $r\in \bar{0}$. Thus there are two classes $\bar{0}, \bar{1}$ and the corresponding ring is $\ZZ_2$.
\end{proof} 

 \begin{corollary}\label{cor.2k}
The set $\uU(\ZZ_{n})$ is a paragon in $\tT(\ZZ_{n})$ if and only if $n=2^k$ for some $k\in \NN$. 

 \end{corollary}
 \begin{proof}
 Since the quotient truss $\tT(\ZZ_{n})/\uU(\ZZ_{n})$ must correspond to a ring of characteristic 2, in the case studied it must correspond to $\ZZ_2$. Thus necessarily we are in the situation of Theorem~\ref{thm.z2}. Hence $\uU(\ZZ_{n})$ is a paragon in $\tT(\ZZ_{n})$ if and only if, for all $m\in \ZZ_n$
 $\gcd(m,n) =1$ or $\gcd(1-m,n) =1$. This is equivalent to $n$ having only even prime factors as needed.
\end{proof}

\begin{example}\label{ex.q}
(1) Let us consider the subring of $\QQ$ of the form
 $$
 \frac{\ZZ}{2\ZZ+1}:=\left\{\frac{n}{2p+1}\; \biggr|\; n,p\in \ZZ\right\}.
 $$
 Observe that the set of all invertible elements of $\frac{\ZZ}{2\ZZ+1}$ is 
 $$
 \uU\left(\frac{\ZZ}{2\ZZ+1}\right)=\oO(\QQ):=\left\{\frac{2q+1}{2p+1}\;\bigr|\; q,p\in \ZZ\right\}.
 $$ 
 Clearly, $\oO(\QQ)$ is not a subring but one can easily check that $\oO(\QQ)$ is a well defined sub-truss (and also a paragon) of $\tT(\QQ).$ The elements $x$ of  $\frac{\ZZ}{2\ZZ+1}$ have either an odd numerator, in which case they are invertible or an even numerator, in which case $1-x$ has an odd numerator, hence invertible.
Note in passing that $\oO(\QQ)$ is an example of a two-sided brace with operation $-+_1-:=[-,1,-].$

(2) Let $\FF$ be a field and consider the local ring $R = \FF[x]/(x^n)$. The polynomials with root 0 are nilpotent hence not invertible. On the other hand polynomials with a constant coefficient are invertible. Explicitly, if $p(x) = \alpha + q(x)$, where $q(x)$ is nilpotent and $\alpha\neq 0$, then 
\begin{equation}\label{inv.pol}
p(x)^{-1} = \alpha^{-1} - \alpha^{-2}\left(q(x)+q(x)^2+\ldots + q(x)^{n-1}\right).
\end{equation}
Hence $R$ satisfies assumptions of Theorem~\ref{thm.z2} and so the set of polynomials with a non-zero constant coefficient is a paragon, and the quotient truss corresponds to the ring $\ZZ_2$.

(3) The situation described in the preceding example can be adapted to polynomial rings with coefficients in general commutative rings. Consider $R=Q[x]/(x^n)$. If $Q$ is not an integral domain, then $R$ does not necessarily have the property of Theorem~\ref{thm.z2}, so neither that the units of $R$ form a paragon nor, in case they do, that the quotient paragon will be associated to $\ZZ_2$ is guaranteed.  By the arguments similar to those in the proof of Lemma~\ref{lem.g.ring} one easily finds that if $\uU(R)$ is a paragon, then so is $\uU(Q)$. Using this information, we now specify $Q=\ZZ_m$, and then by Corollary~\ref{cor.2k}, necessarily $m=2^k$. Units of $\ZZ_{2^k}[x]/(x^n)$ are exactly all polynomials $p(x)$ such that $p(0)$ is coprime with $2$. Clearly, if $p(x)$ is a unit, then $p(0)$ must be a unit, hence coprime with $2$. In the converse direction, the formula \eqref{inv.pol} gives the inverse to any polynomial with the constant term that is a unit in $\ZZ_{2^k}$. Let $p(x)$ be any element of  $\ZZ_{2^k}[x]/(x^n)$. If $p(0)$ is coprime with 2, then $p(x)$ is a unit. Otherwise, $1-p(x)$ has a constant term coprime with 2, hence it is a unit. Thus assertions of Theorem~\ref{thm.z2} are satisfied and we conclude that $\uU(\ZZ_{2^k}[x]/(x^n))$ is a paragon in $\tT(R)$ and the quotient truss is associated to $\ZZ_{2}$.
 \end{example}
 
\begin{lemma}
Let $T$ be a unital truss. If the set of units $\uU(T)$ is a sub-heap of $T$, then $\uU(T)$ is a brace.
\end{lemma}
\begin{proof}
Assume that $\uU(T)$ is a sub-heap and since $\uU(T)$ is a group with truss multiplication,  $\uU(T)$ is a brace-type truss in which every element is invertible, and hence a brace  by \cite[Corollary~3.10]{Brz:par}.
\end{proof}

We conclude this section with the derivation of abelian cyclic braces of \cite[Proposition~4]{Rum:cla} as quotients of a commutative truss by a paragon.
\begin{proposition}\label{prop.z,mod2n}
Let $a$ be a positive integer and let $\ZZ^{(a)}$ denote the commutative unital truss with the heap operation derived from the addition in $\ZZ$, and the multiplication, 
\begin{equation}\label{z.a}
m\cdot n = amn+m+n, \qquad \mbox{for all $m,n\in \ZZ$};
\end{equation}
see \cite[Corollary~3.53]{Brz:par}. 
\begin{zlist}
\item For all $N\in \ZZ_+$, 
$$
N\ZZ = \{mN\;|\; m\in \ZZ\},
$$ 
is a paragon in $\ZZ^{(a)}$.
\item For all $k\geq 1$, $\ZZ^{(2)}/2^{k+1}\ZZ$ is a brace-type truss in which every element is a unit (hence a brace) and
$$
\uU\left(\ZZ^{(2)}/2^{k+1}\ZZ\right) = C_2\oplus C_{2^k}.
$$
\end{zlist}
\end{proposition}
\begin{proof}
(1) $N\ZZ$ is an abelian subgroup of $\ZZ$, hence a sub-heap of $\ZZ$. Note that 0 is the identity in $\ZZ^{(a)}$. Hence, for all $m\in \ZZ$ and $nN\in N\ZZ$,
$$
\begin{aligned}
m\cdot_0(nN) &= [amnN +m +nN, m,0]\\
& = amnN +m +nN-m = (amn+n)N \in N\ZZ.
\end{aligned}
$$
Since $0\in N\ZZ$, the assertion follows.

(2) First, one easily proves by induction that, for all $m\in \ZZ$ and $n\in \NN$,
\begin{equation}\label{power}
m^{\cdot n} = \frac{(am+1)^n-1}{a},
\end{equation}
where $m^{\cdot n}$ means the $n$-th power with respect of the product \eqref{z.a} in $\ZZ^{(a)}$. Using formula \eqref{power} one proves that, for all $k\geq 1$,
\begin{equation}\label{order}
m^{\cdot 2^k} \equiv 0 \!\!\!\mod\! 2^{k+1},
\end{equation}
in $\ZZ^{(2)}$. Indeed, if $k=1$,
$$
m^{\cdot 2} = \frac{(2m+1)^2-1}{2}= 2m(m+1) \equiv 0 \!\!\!\mod\! 4.
$$
Next note that
$$
m^{\cdot 2^{k+1}} = \frac{(2m+1)^{2^{k+1}}-1}{2} =  \frac{(2m+1)^{2^{k}}-1}{2}((2m+1)^{2^{k}}+1).
$$
Hence if the first factor is divisible by $2^{k+1}$, then $
m^{\cdot 2^{k+1}}$ is divisible by $2^{k+2}$, since the second factor is even. Thus the stated congruence relation follows for all $k$ by the principle of induction.

Since 0 is the identity in $\ZZ^{(2)}$, the congruence relation \eqref{order} implies that every element in $\ZZ^{(2)}/2^{k+1}\ZZ$ is a unit and that  elements of $\uU\left(\ZZ^{(2)}/2^{k+1}\ZZ\right)$ have order not greater that $2^k$, hence  $\uU\left(\ZZ^{(2)}/2^{k+1}\ZZ\right)$ is not a cyclic group. We will show that $1$ has the maximal order $2^k$. Since
$$
1^{\cdot 2^k} = \frac{3^{2^k}-1}{2},
$$
this is equivalent to the statement that $3$ is an order $2^k$ element in the group of units $\uU(\ZZ_{2^{k+2}})$. This follows from the (inductively proven) fact that, for all $k$,  
$$
3^{2^k} = 1+n_k 2^{k+2},
$$
where $n_k$ is odd and the observation that the order of any element of $\uU(\ZZ_{2^{k+2}})$ is a power of 2.

Thus, $\uU\left(\ZZ^{(2)}/2^{k+1}\ZZ\right)$ is an abelian group of order $2^{k+1}$ that is not a cyclic group, but contains an element of order $2^k$, hence it must be isomorphic to $C_2\oplus C_{2^k}$.
\end{proof}

\section{Extensions}\label{sec.ext}

By the standard construction, given a ring $R$ and an $R$-bimodule $M$ one can define an extension of $R$ by $M$ as a ring with the abelian group structure $R\oplus M$ and multiplication $(r,m)(r',m') = (rr',rm'+mr')$. In this section we show that this construction can be extended to one-sided modules, but then the result is a truss rather than a ring. 

Let us start with the following motivating observations. 
\begin{example}\label{ex.ext.group}
Let $G$ be an abelian group. Then the ring $\End{}{G}$ acts on $G$ by evaluation, i.e.\
$\End{}{G}\times G\to G $, $(f,g)\mapsto f(g).$  One easily checks that the following binary operation on $\End{}{G}\oplus G$
\begin{equation}\label{end:mult}
(f,g)(f',g')=(f\circ f',g+f(g')), \qquad \mbox{for all $f,f'\in \End{}{G}$, $g,g'\in G$},
\end{equation}
is associative. However, the operation \eqref{end:mult} does not distribute over the addition $\End{}{G}\oplus G$, since for all $f,f',f''\in \End{}{G}$ and $g,g',g''\in G$, on one hand
$$
(f,g)((f',g')+(f'',g''))=(f\circ f'+f\circ f'',g+f(g')+f(g'')),
$$
while on the other
$$
(f,g)(f',g')+(f,g)(f'',g'')= (f\circ f'+f\circ f'',g+g+f(g')+f(g'')).
$$
Notwithstanding, it is easy to check that the operation \eqref{end:mult} distributes over the ternary heap operation associated to the addition in $\End{}{G}\oplus G$. In summary, the extension of the endomorphism ring of a group by this group is a truss.
\end{example}

Now we place Example~\ref{ex.ext.group} in a more general framework of extensions of trusses by one-sided modules. 

\begin{theorem}\label{thm.ext}
Let $T$ be a truss and let $M$ be a left $T$-module. Then, for all $e\in {M},$ $T\times M$ is a truss with the Cartesian product heap structure and multiplication
\begin{equation}\label{prod.ext}
(t,m)(t',m')=(tt', [m,t\cdot e,t\cdot m']), 
\end{equation}
for all $t,t'\in T$ and $m,m'\in M$. We denote this truss by $T[M;e]$ and call it an \textbf{extension of $T$ by $M$}.
\end{theorem}
\begin{proof}
That $T\times M$ with given operations is a truss can be checked by direct calculations. We start with the associative law. For all $t,t',t''\in T$ and $m,m',m''\in M$,
$$
\begin{aligned}
(t,m)((t',m')(t'',m''))&=(t,m)\left(t't'',[m',t'\cdot e,t'\cdot m'']\right)\\
&=\left(tt't'', [m,t\cdot e,t\cdot [m',t'\cdot e,t'\cdot m'']]\right)\\ 
&=\left(tt't'', [[m,t\cdot e,t\cdot m'],tt'\cdot e,tt'\cdot m'']\right)\\ 
&=(tt',[m,t\cdot e,t\cdot m'])(t'',m'')=((t,m)(t',m'))(t'',m''),
\end{aligned}
$$
where the third equality follows by the distributive and associative laws for modules over trusses and by the associativity of the heap operation. To prove the left distributive law we compute, for all $t,t',t'', t'''\in T$ and $m,m',m'',m'''\in M$: 
$$
\begin{aligned}
(t,m)[(t',m'),&(t'',m''),(t''',m''')] =(t,m)([t',t'',t'''],[m',m'',m'''])\\ 
&=(t[t',t'',t'''],[m,t\cdot e,t\cdot [m',m'',m'''])\\ 
&=([tt',tt'',tt'''], [[m,m,m],[t\cdot e,t\cdot e,t\cdot e],[t\cdot m',t\cdot m'',t\cdot m''']])\\ 
&=([tt',tt'',tt'''],[[m,t\cdot e,t\cdot m'],[m,t\cdot e,t\cdot m''],[m,t\cdot e,t\cdot m''']])\\ 
&=[(t,m)(t',m'),(t,m)(t'',m''),(t,m)(t''',m''')].
\end{aligned}
$$
The third equality follows by the distributive laws for trusses and modules over trusses and by the Mal'cev identities (which imply that the heap operation is an idempotent operation). The rearrangement of brackets leading to the fourth equality is possible since $M$ is an abelian heap.
Similarly, for the right distributivity,
$$
\begin{aligned}[]
[(t',m'),&(t'',m''),(t''',m''')](t,m) =\left([t',t'',t'''],[m',m'',m''']\right)(t,m)\\
 &=([t't,t''t,t'''t], [[m',m'',m'''],[t',t'',t''']\cdot e ,[t',t'',t''']\cdot m])\\ 
 &=([t't,t''t,t'''t], [[m',m'',m'''],[t'\cdot e,t''\cdot e,t'''\cdot e] ,[t'\cdot m,t''\cdot m,t'''\cdot m]])\\ 
   &=([t't,t''t,t'''t], [[m',t'\cdot e,t'\cdot m],[m'',t''\cdot e,t''\cdot m],[m''',t'''\cdot e,t'''\cdot m]])\\ 
&= [(t',m')(t,m),(t'',m'')(t,m),(t''',m''')(t,m)].
\end{aligned}
$$
Here, as in the preceding computation, the third equality is obtained by the distributive laws, while the fourth one follows from the fact that $M$ is an abelian heap.
Since the operation \eqref{prod.ext} is associative and distributes from both sides over the heap operation in $T\times M$, $T[M;e]$ is a truss, as claimed.
\end{proof}

A natural question that arises here is whether $T[M;e]$ can be a truss  associated with a ring.

\begin{lemma}\label{lem.ext.ring}
The truss $T[M;e]$  is ring-type if and only if $M=\{e\}$ and $T=\tT(R)$ for some ring $R$.
\end{lemma}
\begin{proof}
Let $T[M;e]$ be a ring-type truss and $(t',m')$ be an absorber in $T[M;e]$. Then, for all $m\in M$ and $t\in T$,
$$
(t't,[m',t'\cdot e,t'\cdot m]) =(t', m')(t,m)= (t',m')=(t, m)(t', m')=(tt',[m,t\cdot  e,t\cdot m']),
$$
which immediately implies that $t'$ is the absorber in $T$. 
Therefore, $T=\tT(R)$, where $R$ has the same multiplication as $T$ and the abelian group structure obtained as the $t'$-retract of $(T,[-,-,-])$.

Observe that $[m,t\cdot e,t\cdot m']=m'$ implies $t\cdot m'=[t\cdot e,m,m']$, so choosing $m = t\cdot e$ we obtain $t\cdot m' = m'$, for all $t\in T$, i.e.\ $m'$ is an absorber in $M$. Hence $m' = [t\cdot e,m,m']$, for all $m\in M$. In particular, for $m=e$, $t\cdot e =e$.
 Therefore, for all  $m\in M$, 
$$
[e,m,m']= m',
$$
 which  implies that $m=e$. So if the truss $T[M;e]$  is ring-type, then   $M=\{e\}$. 
 
 The converse implication follows by the simple observation that $\tT(R)\cong \tT(R)[\{e\};e]$.
\end{proof}

Put differently, Lemma~\ref{lem.ext.ring} asserts that the truss obtained by extension by a non-trivial module is never a truss associated to a ring.

We now list properties of an extension truss.
\begin{theorem}\label{thm.mod}
Let $T$ be a truss, $M$ be a left $T$-module and let $e\in M$.
\begin{zlist}
\item For any $\bar{e}\in M$, $T[M;e]\cong T[M;\bar{e}]$.
\item  $M$ is a left  $T[M;e]$-module with the action, for all $m,m'\in M$ and $t\in T$,
$$
(t,m)\cdot m'=[m,t\cdot e,t\cdot m'].
 $$
In particular, $(t,m)\cdot e =m$.
\item The induced actions of $T[M;e]$ on $M$ coincide with the induced actions of $T$ on $M$, i.e.\ for all ${\bar{e}\in M}$
$$
(t,m)\cdot_{\bar{e}}m'= t\cdot_{\bar{e}}m'.
$$ 
In particular, if $\bar{e}$ is an absorber in the $T$-module ${M}$, then  $(t,m)\cdot_{\bar{e}}m'=t\cdot m'$.
 \item \label{to.truss} For all $a\in T$, the sub-heap 
 $
 M_a := \{a\}\times M$ 
 is a paragon in $T[M;e]$. Furthermore,
 $$
 T[M;e]/{M_a} \cong T.
 $$
 $M_a$ is an ideal in $T[M;e]$ if and only if $a$ is an absorber in $T$.
 \item \label{to.mod} The sub-heap 
 $T_e := T\times \{e\}$ is a sub-truss and a left paragon of $T[M;e]$. Furthermore,
 $$
 T[M;e]/{T_e} \cong M,
 $$
 as left $T[M;e]$-modules.
  \item \label{unitality}
The extension truss $T[M;e]$ is unital if and only if $T$ is a unital truss and $M$ is a unital module. Furthermore, $\uU(T[M;e]) = \uU(T)\times M$.
\end{zlist}
\end{theorem}
\begin{proof}
(1) Consider the heap automorphism:
$$
\Theta = \id_T\times \tau_e^{\bar{e}} : T\times M\lra T\times M, \qquad (t,m)\lto (t,[m,e,\bar{e}]).
$$
We will show that $\Theta$ is a truss isomorphism from $T[M;e]$ to $T[M;\bar{e}]$. Let us take any $t,t'\in T$ and $m,m'\in M$, and compute
$$
\begin{aligned}
\Theta(t,m)\Theta(t',m') &= \left( t, \left[m,e,\bar{e}\right]\right)\left( t', \left[m',e,\bar{e}\right]\right)\\
&= \left(tt', \left[\left[m,e,\bar{e}\right], t\cdot \bar{e}, t\cdot \left[m',e,\bar{e}\right]\right]\right)\\
&= \left(tt',\left[\left[m,e,\bar{e}\right],t\cdot \bar{e},  \left[t\cdot m',t\cdot e,t\cdot \bar{e}\right]\right]\right)\\
&= \left(tt',\left[\left[t\cdot m',t\cdot e,t\cdot \bar{e}\right],t\cdot \bar{e},  \left[m,e,\bar{e}\right] \right]\right)\\
&= \left(tt', \left[\left[t\cdot m',t\cdot e, m\right],e,\bar{e}\right]\right)\\
&= \left(tt', \left[\left[m,t\cdot e, t\cdot m'\right],e,\bar{e}\right]\right) = \Theta\left((t,m)(t',m')\right),
\end{aligned}
$$
where the third equality follows by the left distributive law for actions. The fourth and sixth equalities are consequences of the fact that $M$ is an abelian heap. The key cancellation and rearrangement of brackets leading to the fifth equality result from the associative laws for and Mal'cev properties of heap operations. Thus $\Theta$ is the required isomorphism of trusses.\smallskip

(2) The proof of the associative and distributive laws for $M$ as a $T[M;e]$-module follow by the same chains of arguments as that in the proof of Theorem~\ref{thm.ext} for  the corresponding laws for the truss $T[M;e]$, and thus are left to the reader. The property $(t,m)\cdot e =m$ follows immediately by the Mal'cev identity.
\smallskip

(3) For the first statement, observe that
$$
\begin{aligned}
(t,m)\cdot_{\bar{e}}m' &=[(t,m)\cdot m',(t,m)\cdot\bar{e},\bar{e}]\\
&=[[m,t\cdot e,t\cdot m'],[m,t\cdot e,t\cdot \bar{e}],\bar{e}]=[t\cdot m',t\cdot \bar{e},\bar{e}] = t\cdot _{\bar{e}} m',
\end{aligned}
$$
by the Mal'cev identities, associativity of the heap operation and by the fact that $M$ is an abelian heap. The second statement follows immediately by the Mal'cev identity. \smallskip

(4) For all $t\in T$, $m,m',m''\in M$,
$$
\begin{aligned}
\left[(t,m)(a,m'), (t,m)(a,e), (a,e)\right] &= \left[\left(ta, \left[m,t\cdot e,t\cdot m'\right]\right), \left(ta, \left[m,t\cdot e,t\cdot e\right]\right), (a,e)\right]\\
&= \left(\left[ta,ta,a\right], \left[\left[m,t\cdot e,t\cdot m'\right], m,e \right]\right)\\
&= \left(a, \left[t\cdot m',t\cdot  e, e\right]\right)\in M_a,
\end{aligned}
$$
by the Mal'cev identities and the associativity of heap operations. Hence $M_a$ is a left paragon. The right paragon property of $M_a$ is proven in a similar way.

Consider the following maps
$$
\begin{aligned}
\varphi &:T\lra T[M;e]/{M_a}, \qquad t\lto \overline{(t,e)},\\
\psi &:  T[M;e]/{M_a} \lra T, \qquad \overline{(t,m)}\lto t.
\end{aligned}
$$
The map $\varphi$ is the quotient of the heap map $\tilde{\varphi}: T\lra T[M;e]$, given by $t\lto (t,e)$. Note that,  for all $t,t'\in T$,
$$
\tilde{\varphi}(t)\tilde{\varphi}(t') =(t,e)(t',e) =(tt', [e,t\cdot  e, t\cdot e]) = (tt',e) = \tilde{\varphi}(tt'),
$$
i.e.\ $\tilde{\varphi}$ is a truss homomorphism, and thus so is $\varphi$.

We need to check whether the map $\psi$ is well-defined. By the definitions of $M_a$ and the sub-heap relation, $(t,m)\sim_{M_a}(t',m')$ if and only if there exist $m'',m'''\in M$ such that
\begin{equation}\label{class}
(a,m''') = \left[(t,m), (t',m'), (a,m'')\right] = \left([t,t',a],[m,m',m'']\right).
\end{equation}
Thus, in particular $a= [t,t',a]$ which implies $t'=t$. Therefore, the element $t$ is fully determined by the class of $(t,m)$. This means that the function $\psi$ is well-defined. The second consequence of \eqref{class} is that the class of $(t,m)$ is fully determined by $t$, i.e.\ it does not depend on the choice of $m$. This implies that the composite function $\varphi\circ \psi$ is the identity. That the composite $\psi\circ\varphi$ is identity is obvious.

The fact that $M_a$ is an ideal in $T[M;e]$ if and only if $a$ is an absorber in $T$ follows immediately from the analysis of the first entries in the products $(t,m)(a,m')$ and $(a,m)(t,m')$. \smallskip

(5) It is obvious that $T_{e}$ is closed under the ternary heap operation. That $T_e$ is closed under the multiplication as well follows immediately by the Mal'cev identity. Let us take $(t'',m)\in T[M;e]$ and $(t',e),(t,e)\in T_e$,  and use the definition of the truss operations for $T[M;e]$ and Mal'cev identities to compute:
$$
\begin{aligned}
{}[(t'',m)(t',e),(t'',m)(t,e),(t,e)] &=[(t''t',[m,t''\cdot e,t''\cdot e]),(t''t,[m,t''\cdot e,t''\cdot e]),(t,e)]\\ 
&=([t''t',t''t,t], [m,m,e]) =([t''t',t''t,t], e)\in T_e.
\end{aligned}
$$
Therefore, $T_{e}$ is a left paragon hence a left induced $T[M;e]$-submodule of $T[M;e]$. The left $T[M;e]$-module isomorphism $T[M;e]/T_e \lra M$ is constructed in the similar way to the isomorphism in part (4). In particular, one finds that, for all $t,t'\in T$ and $m,m'\in M$,
\begin{blist}
\item $(t,e)\sim_{T_e}(t',e)$ and
\item if $(t,m)\sim_{T_e}(t',m')$, then $m=m'$.
\end{blist}
Hence we can fix $a\in T$ and consider the following functions:
$$
\begin{aligned}
\varphi &:M\lra T[M;e]/{T_e}, \qquad m\lto \overline{(a,m)},\\
\psi &:  T[M;e]/{T_e} \lra M, \qquad \overline{(t,m)}\lto m.
\end{aligned}
$$
The map $\varphi$ is the quotient of a heap homomorphism $m\lto (a,m)$ hence a heap homomorphism. Furthermore, using observation (a) one can compute
$$
\begin{aligned}
(t,m)\cdot \varphi(m') &= \overline{(t,m) (a,m')}
 = \overline{(ta, [m,t\cdot e,t\cdot m'])}\\
 &= \overline{(a, [m,t\cdot e,t\cdot m'])} = \overline{(a, (t,m)\cdot m')} = \varphi((t,m)\cdot m').
 \end{aligned}
 $$
 Hence $\varphi$ is a homomorphism of $T[M;e]$-modules. That $\psi$ is well-defined follows by the observation (b). Again (a) implies that $\varphi\circ\psi = \id$ and the other inverse property is obvious.
\smallskip 

(6) If $T$ has identity 1, then $(1,e)$ is the identity for the extended truss $T[M;e]$ by the Mal'cev properties and by the unitality of $M$. Conversely, if $(a,\bar{e})$ is the identity of $T[M;e]$, then, for all $(t,m) \in T[M;e]$,
\begin{subequations}\label{id}
\begin{equation}\label{id.1}
(t,m) = (a,\bar{e})(t,m) = (at, [\bar{e},a\cdot e,a\cdot m]),
\end{equation}
\begin{equation}\label{id.2}
(t,m) = (t,m)(a,\bar{e}) = (ta, [m,t\cdot e,t\cdot\bar{e}]).
\end{equation}
\end{subequations}
Comparison of the first elements in each pair in equalities \eqref{id} yields that $T$ is unital with the identity $1=a$. Evaluation of \eqref{id.1} at $m=e$ produces the equality $\bar{e} =e$, while its evaluation at $m=1\cdot e$ gives $\bar{e} = 1\cdot e$, and hence $1\cdot m =m$, for all $m\in M$, again by \eqref{id.1}.

Let $u$ be a unit in $T$. Then, for all $m\in M$,
$$
\begin{aligned}
(u,m)(u^{-1}, [e, u^{-1}\!\cdot m, u^{-1}\!\cdot e]) &= (uu^{-1}, [m, u\cdot e, u\cdot [e, u^{-1}\!\cdot m, u^{-1}\!\cdot e]])\\
&= (1,[[m, u\cdot e,  u\cdot e], uu^{-1}\!\cdot m, uu^{-1}\!\cdot e]) ] = (1,e),
\end{aligned}
$$
by the distributive and associative laws for modules, axioms of heaps and unitality of $M$. In a similar way, by the axioms of a heap
$$
\begin{aligned}
(u^{-1}, [e, u^{-1}\!\cdot m, u^{-1}\!\cdot e])(u,m) = (u^{-1}u, [[e, u^{-1}\!\cdot m, u^{-1}\!\cdot e],  u^{-1}\!\cdot e, u^{-1}\!\cdot m]) = (1,e).
\end{aligned}
$$
Hence, if $u$ is a unit in $T$, $(u,m)$ is a unit in $T[M;e]$, for all $m\in M$.
This proves the inclusion $\uU(T)\times M\subseteq \uU(T[M;e])$. The converse inclusion follows immediately from the definition of the product in $T[M;e]$.
\end{proof}

\begin{remark}\label{rem.ext.seq}
Assertions \eqref{to.truss} and \eqref{to.mod} of Theorem~\ref{thm.mod} yield the following  sequence, for all $a\in T$,
$$
\xymatrix{ M\ar@{^{(}->}[rr]^-{\iota_a} & & T[m;e] \ar@{->>}@<.7ex>[rr]^-{\pi} && T\ar@{_{(}->}@<.8ex>[ll]^-{j} ,
}
$$
where $\iota_a: m\lto (a,m)$, $j: t\lto (t,e)$ and $\pi: (t,m)\lto t$. This sequence is a split-exact sequence of trusses in the following sense. The map $\pi$ is a split epimorphism of trusses ($j$ is the splitting monomorphism) and the relation induced by the image of $\iota_a$ is the kernel relation for $\pi$. In summary, Theorem~\ref{thm.ext} describes a \textbf{split extension of trusses}.
\end{remark}

The assertion \eqref{unitality} of Theorem~\ref{thm.mod} implies the following
\begin{corollary}\label{cor.ext.brace}
An extension truss $T[M;e]$ is a truss associated to a brace if and only if $T$ is associated to a brace and $M$ is a unital $T$-module.
\end{corollary}
\begin{proof}
By Theorem~\ref{thm.mod}~\eqref{unitality} $T[M;e]$ is brace-type (unital) if and only if $T$ is brace-type. Furthermore, $\uU(T[M;e]) = T[M;e]$ if, and only if $\uU(T) = T$, i.e.\ $T[M;e]$ is a multiplicative group (brace) if and only if $T$ is as well.
\end{proof}

It might be instructive to contrast Lemma~\ref{lem.ext.ring} with Corollary~\ref{cor.ext.brace}. While only trivial truss extension of a ring results in a ring (alas an extension in name only), an extension of a brace by any unital module is a brace.

\begin{example}\label{ex.ext.brace}
Let $B$ be a brace. Then $\tT(B)$ is a left module over itself by multiplication, hence one can consider the extension truss $\tT(B)[\tT(B);1]$. As a heap,
$$
\tT(B)[\tT(B);1] = \tT(B)\times \tT(B).
$$
The multiplication comes out as, for all $b_1,b_2, b'_1,b'_2\in B$,
\begin{equation}\label{prod.brace}
(b_1,b_2)(b'_1,b'_2) = (b_1b'_1, b_2 -b_1 + b_1b_2').
\end{equation}
$\tT(B)[\tT(B);1]$ is the truss associated to the brace with additive structure given by $B\oplus B$ and multiplication given by the formula \eqref{prod.brace}.  Note that, even if the brace $B$ is abelian (i.e.\ the truss $\tT(B)$ is commutative), the extended brace need not be so. 

For an explicit example we may consider the brace obtained as the $0$-retract of the truss $\ZZ^{(2)}/2^{k+1}\ZZ$ in Proposition~\ref{prop.z,mod2n}. The multiplication in $\ZZ^{(2)}/2^{k+1}\ZZ[\ZZ^{(2)}/2^{k+1}\ZZ;0]$ is given by
$$
(m,s)(n,t) = \left(2mn +m +n \!\! \mod 2^{k+1}\, ,\, 2mt +s+t \!\! \mod 2^{k+1}\right).
$$
In particular, the multiplicative 
group of $\ZZ^{(2)}/4\ZZ[\ZZ^{(2)}/4\ZZ;0]$ is generated by elements
$$
a=(0,1), \qquad x=(1,0), \qquad y=(2,0),
$$
which satisfy the following relations
$$
a^4 = x^2 = y^2 = (0,0), \quad xax=a^3, \quad xy=yx, \quad ay=ya,
$$
and hence it is isomorphic to the direct product of the dihedral group $D_8$ and the cyclic group $C_2$. The additive structure of the brace associated to $\ZZ^{(2)}/4\ZZ[\ZZ^{(2)}/4\ZZ;0]$ is that of $C_4\oplus C_4$.

Since the socle, $\mathrm{Soc}(B)$, is an ideal in a brace $B$, it is a left (induced) $T(B)$-module. The multiplication on $\tT(B)[\mathrm{Soc}(B);1] = \tT(B)\times \mathrm{Soc}(B)$ derived from \eqref{prod.brace} reduces to
$$
(b,a)(b',a') = (bb', a-b+ba) = (bb', a+bab^{-1}).
$$
\end{example}

\begin{remark}\label{rem.iter}
In view of Theorem~\ref{thm.mod}, since $M$ is a module over the extended truss $T[M;e]$, the extension procedure could be iterated. This, however, rather than producing genuinely new examples boils down to the truss extension by the $T$-module obtained as the product of $T$-modules. 
\end{remark}

We conclude the paper with the following
\begin{remark}\label{rem.left}
By analysing the proofs of main  statements of this section, that is Theorem~\ref{thm.ext} and Theorem~\ref{thm.mod}, one can easily convince oneself that the assertions hold for left trusses and their left modules (and also for right trusses if left modules are replaced by right ones). Since for a left truss $T$ no right distributivity is assumed, one does not assume that the action of $T$ on a left module $M$ right distributes over the heap operation on $T$ (otherwise $T$ would not be its own module). But this right distributivity is not needed neither for the associativity of the product in \eqref{prod.ext} in $T[M;e]$ nor for its left distributivity. Thus, if $T$ is a left truss and $M$ is a left $T$-module, then $T[M;e]$ with Cartesian product heap structure and with multiplication \eqref{prod.ext} is a left truss. Main assertions of Theorem~\ref{thm.ext} stand if  the words ``truss'' or ``paragon'' are qualified by the adjective ``left''. Most importantly, the one-sided version of Corollary~\ref{cor.ext.brace}, with no changes in the formula for the multiplicative group structure, equips one with the procedure of obtaining left braces from left braces. 

For example, if $(B,+,\cdot)$ is a left brace and $(B,\cdot)$ acts on an abelian group $M$ by automorphisms, i.e.\ $M$ is a left module over a brace in the sense of Rump (see e.g.\ \cite[Section~5]{Rum:gro}),  then $M$ is a left $\tT(B)$-module, and  by Corollary~\ref{cor.ext.brace} $T(B)[M;0]$ is a left truss associated to a left brace obtained as a special case of \cite[Theorem~3.3]{Bac:ext} (with the trivial cocycles and the right action of  $(B,\cdot)$ on $M$ given by the inverse of the left action). 
Note that not all modules over the truss $\tT(B)$ associated to a brace $B$ correspond to modules  over this brace, which indicates that the results presented in this section provide one with additional flexibility, and thus might lead to new examples of braces.
\end{remark}

\section*{Acknowledgement}
We would like to thank Martin Crossley for an interesting discussion about group structures in Proposition~\ref{prop.z,mod2n}.

\end{document}